\documentclass[10pt,a4paper,reqno]{amsart}

\usepackage{amsfonts,amsthm,latexsym,amsmath,amssymb,amscd,amsmath, mathrsfs, color, epsf}
\usepackage{graphicx}

\makeatletter
\@namedef{subjclassname@2020}{%
  \textup{2020} Mathematics Subject Classification}
\makeatother

\newtheorem{theorem}{Theorem}
\newtheorem{proposition}[theorem]{Proposition}
\newtheorem{lemma}[theorem]{Lemma}
\newtheorem{definition}{Definition}
\newtheorem{corollary}{Corollary}
\newtheorem{example}{Example}
\newtheorem{remark}{Remark}

\newtheorem{conjecture}{Conjecture}

\newtheorem{notation}{Notation}
\newtheorem{problem}{Problem}

\newcommand{\CC}{\mathcal{C}}

\newcommand{\bR}{\mathbb{R}}
\newcommand{\bC}{\mathbb{C}}
\newcommand{\bZ}{\mathbb{Z}}

\newcommand{\ee}{\end{equation}}
\newcommand {\al}{\alpha}

\newcommand{\si}{\sigma}

\newcommand{\D} {\mathcal D}

\newcommand {\OO} {\mathcal O}

\begin{document}

\numberwithin{equation}{section}

\title[In search of  Newton-type inequalities]{In search of  Newton-type inequalities}
\author[O.~Katkova]{Olga Katkova}

\address{Department of Mathematics, University of Massachusetts Boston}
\email{woman.math@gmail.com}

\author[B. Shapiro]{Boris Shapiro}

\address{   Department of Mathematics,     Stockholm University,        S-10691, Stockholm, Sweden}
\email{shapiro@math.su.se}

\author[A.~Vishnyakova]{Anna Vishnyakova}

\address{ Department of Mathematics,  Holon Institute of Technology, 
Holon, Israel}
\email{annalyticity@gmail.com}

\date{\today}

\maketitle

\dedicatory {\hskip2cm To our friend and coauthor Vlado Petrov Kostov}

\begin{abstract}
In this paper, we prove a number of results providing either necessary or sufficient 
conditions guaranteeing that the number of real roots  of real polynomials of a given 
degree  is either less or greater than a given number. We also provide counterexamples 
to two earlier conjectures refining Descartes rule of signs. 
\end{abstract}

\maketitle

\smallskip

\section{Introduction}

The main motivation for the present paper comes from  the classical Newton inequalities and 
Hutchinson's theorem, see \cite{Hu,Ku} together with three  conjectures formulated in 
\cite{Sh}, see \S~\ref{sec:app2} and  \cite{FNS}.

\medskip
Our present set-up is as follows. 

\begin{notation} 
Denote by $\bR[x]$ the linear space of all polynomials with real coefficients, and by $\bR^{+}[x]$ the  cone of all polynomials with positive coefficients. Denote by $\bR_n[x]$ the (affine) space of monic degree $n$ polynomials with real coefficients and by $\bR^+_n[x]\subset \bR_n[x]$  the orthant of monic polynomials of degree $n$ with positive coefficients.   

Denote by $\widehat \bR_n[x]$ the (affine) space of monic degree $n$ polynomials with real coefficients and constant term $1$ and by $\widehat \bR^+_n[x]\subset \bR_n[x]$  the orthant of monic polynomials of degree $n$ with positive coefficients and constant term $1$.  

\smallskip
Now let $\Delta_n$ denote the discriminant of the family of polynomials $x^n+a_{n-1}x^{n-1}+\cdots +a_0,$  i.e., $\Delta_n \in \bZ[a_{n-1}, \dots, a_{0}]$  is the unique (up to constant factor) irreducible polynomial such that $x^n+a_{n-1}x^{n-1}+\cdots +a_0$   has a root of multiplicity greater than $1$ implies that $\Delta_k(a_{n-1},\dots, a_{0}) = 0$. Denoting by $\mathfrak D_n\subset \bR_n[x]$ the zero locus of $\Delta_n$, we will be mainly interested in its subset $\D_n \subset \mathfrak D_n$ 
consisting of polynomials in $\bR_n[x]$ having a {\it multiple real zero}. (It is well-known  that $\mathfrak D_n\setminus \D_n$ has codimension $1$ in $\mathfrak D_n$. Indeed the complement $\mathfrak D_n\setminus \D_n$  consists of real monic polynomials of degree $n$ having at least one pair of complex conjugated roots each having multiplicity at least $2$ and therefore has real codimension $2$ in $\bR_n[x]$ and real codimension $1$ in $\mathfrak D_n$). 

Further, let $\D^+_n\subset \bR^+_n[x]$ denote the {\it positive discriminantal locus} which is the set of monic polynomials of degree $n$ with positive coefficients having {\it at least one real multiple root}. In other words, $\D^+_n$ is the restriction of $\D_n$ to $\bR^+_n[x]$. Finally, let $\widehat \D_n$ be the restriction of $\D_n$ to $\widehat \bR_n[x]$ and $\widehat \D_n^+$ be the restriction of $\D_n$ to $\widehat \bR_n^+[x]$.  

\medskip
For each polynomial $P(x)=\sum_{k=0}^{n} a_k x^k $ with positive coefficients, we define the sequence of (positive) numbers\\

\begin{equation}
\label{a1}
q_k=q_k(P)=\frac{a_k^2}{a_{k-1}a_{k+1}}, \; k=1,\dots, n-1. 
\end{equation}

\end{notation}

In what follows we will try to generalize and reinterpret  the following classical results of J.~I.~Hutchinson \cite{Hu} and I.~Newton, see also Proposition~\ref{prop:old} below.\\ 

\noindent
{\bf Theorem H.} An entire function $p(x) = a_0 +a_1 x+\dots +a_n x^n +\dots$ with strictly positive coefficients has the property that all of its finite segments $a_i x^i+a_{i+1}x^{i+1}+...+a_jx^j$  have only real roots if and only if  for $k = 1, 2, \dots$, 

\begin{equation}
\label{eq:Hu}
q_k\ge 4. 
\end{equation}

\medskip

\noindent
{\bf Theorem N.} 
Let $P(x)=\sum_{k=0}^{n} a_k x^k  \in \mathbb{R}_{+}[x]$ be a  polynomial with all real roots. Then for $k=1,2,\ldots, n-1,$

\begin{equation}
\label{a9}
q_{k} \geq \frac{k+1}{k}\cdot \frac{n-k+1}{n-k}.
\end{equation}


\smallskip
An  easy to prove statement is that for any positive integer $n$, the set $\bR_n[x] \setminus \D_n$ of all real monic degree $n$ polynomials with all simple real roots consists of  $\left[\frac {n}{2}\right] +1$ contractible connected components enumerated by the number of simple real zeros of polynomials belonging to the respective component.

\smallskip
Similar results hold for   $\bR_k^+[x]$ and $\widehat \bR_k^+[x]$, see Lemma~\ref{lm:compl} and Corollary~\ref{cor:compl} below.

\smallskip
\begin{notation}
Fix a  sequence of (strict) inequality signs $\bar \si=(\si_1,\si_2,\dots, \si_{n-1})$ of length $n-1$, where each entry is either $<$ or $>$. (One can interpret $\bar \si$ as a binary sequence).  Further let ${\bar \epsilon}=(\epsilon_1,\epsilon_2,\dots, \epsilon_{n-1})$ be a sequence of positive numbers.  Given a pair $(\bar \si, \bar \epsilon)$,  define the subset $\widehat \bR^+_{\bar \si, \bar \epsilon} \subset \widehat \bR^+_n[x]$ as the set of all polynomials $P(x)=x^n+a_{n-1}x^{n-1}+\dots +a_1 x+ 1$ satisfying the system of inequalities $$\{q_j \lessgtr_{\bar \si}  \epsilon_j\},\; j=1,\dots, n-1$$ where the inequality sign in defined by $\si_j$. 

\smallskip
Consider the map $\text{Log} |\cdot |: \widehat \bR^+_n[x] \to \bR^{n-1}$ sending the polynomial $P(x)=x^n+a_{n-1}x^{n-1}+a_{n-2}x^{n-2}+\dots + a_{1}x+1$ with positive coefficients to the $(n-1)$-tuple 
$(\log a_1, \log a_2, \dots ,$ $ \log a_{k-1})$. (The map $\text{Log}$ is a diffeomorphism between the source and the target spaces). Denote the coordinates in the image space $\bR^{n-1}$  by $(\al_1,\al_2,\dots, \al_{n-1})$. 

\end{notation}

\begin{remark} {\rm Observe that  logarithmizing the inequality in the Hutchinson theorem we obtain the linear inequality 
$$2\al_k-\al_{k-1}-\al_{k+1}\ge \ln 4$$ and logarithmizing Newton's inequality we get 
$$2\al_k-\al_{k-1}-\al_{k+1}\ge \ln ((k+1)(n-k+1))-\ln (k(n-k)).$$
}
\end{remark} 

\medskip
The prototype result which we want to generalize in this paper is as follows. Denote by $\Sigma_n^+\subset \bR_n^+[x]$ the set of all degree $n$ real-rooted polynomials with positive coefficients. 

\smallskip
\begin{proposition}[see Proposition 7 of \cite{KoSh}]\label{prop:old}
 
 \noindent
 {\rm (i)} The polyhedral cone described by (the logarithm of) Hutchinson’s inequalities \eqref{eq:Hu}  is the maximal polyhedral cone contained in $\text{Log}\,(\Sigma_n^+)$. 
 
 \smallskip
 \noindent
 {\rm (ii)} The minimal polyhedral cone  containing $\text{Log}\,(\Sigma_n^+)$ is given by (the logarithm of) Newton’s inequalities.
\end{proposition}

\medskip
\begin{notation} For $n$ even, consider two filtrations of $\bR_n[x]$, 
$$\OO_{< 2}\subset \OO_{< 4} \subset \dots \subset \OO_{< n}\subset\bR_n[x]\supset \OO_{\ge 2}\supset \OO_{\ge 4} \dots \supset \OO_{\ge n},$$
where $\OO_{< 2\ell}$ (resp. $\OO_{\ge 2\ell}$) is the  set   of all polynomials in $\bR_n[x]$ with fewer than (resp. at least) $2\ell$ real roots (counting multiplicities).

Analogously, 
$$ \widehat{\OO}_{< 2}\subset \widehat{\OO}_{< 4} \subset \dots \subset \widehat{\OO}_{< n}\subset\widehat{\bR}_n[x]\supset \widehat{\OO}_{\ge 2}\supset \widehat{\OO}_{\ge 4} \dots \supset \widehat{\OO}_{\ge n}$$
is the restriction of the above filtrations to  $\widehat{\bR}_n[x]$; 

$$ \OO_{< 2}^+\subset \OO_{< 4}^+ \subset \dots \subset \OO_{< n}^+\subset\bR_n^+[x]\supset \OO_{\ge 2}^+\supset \OO_{\ge 4}^+ \dots \supset \OO_{\ge n}^+$$
is the restriction of the above filtrations to  ${\bR}_n^+[x]$; 

$$ \widehat{\OO}_{< 2}^+\subset \widehat{\OO}_{< 4}^+ \subset \dots \subset \widehat{\OO}_{< n}^+\subset\widehat{\bR}_n^+[x]\supset \widehat{\OO}_{\ge 2}^+\supset \widehat{\OO}_{\ge 4}^+ \dots \supset \widehat{\OO}_{\ge n}^+$$
is the restriction of the above filtrations to  $\widehat{\bR}_n^+[x]$. 

\smallskip
Similarly, for $n$ odd, we have the following two filtrations of $\bR_n[x]$, 
$$ \OO_{< 3}\subset \OO_{< 5} \subset \dots \subset \OO_{< n}\subset \bR_n[x]\supset \OO_{\ge 3}\supset \OO_{\ge 5} \dots \supset \OO_{\ge n},$$
where $\OO_{< 2\ell+1}$ (resp. $\OO_{\ge 2\ell+1}$) is the  set   of all polynomials in $\bR_n[x]$ with fewer than (resp. at least) $2\ell+1$ real roots (counting multiplicities) and their restrictions to $\widehat{\bR}_n[x]$, ${\bR}_n^+[x]$, and $\widehat{\bR}_n^+[x]$. 
 
\end{notation}

\begin{remark} {\rm One can easily see that the terms of the left filtration are open while the terms of the right filtration are closed in $\bR_n[x], \widehat {\bR}_n[x], \bR_n^+[x]$, and $\widehat {\bR}^+_n[x]$ resp. Furthermore,  the first and the second filtrations are complementary in the following sense. For any positive integer $n$ and any $2\le \ell\le n,\; \ell \equiv n \mod 2$ one has, 
$$\OO_{< \ell} \cup \OO_{\ge \ell}=\bR_n[x]\quad  \text{and} \quad \OO_{< \ell} \cap \OO_{\ge \ell}=\emptyset.$$  The intersection $H_\ell=\overline{\OO}_{< \ell}\cap \OO_{\ge \ell}$ is a semi-algebraic hypersurface in $\bR_n[x]$ consisting of all polynomials having exactly $\ell$ real roots counting multiplicities with at least one non-simple real root. (Here $\overline{\OO}_{< \ell}$ stands for the closure of the open 
domain ${\OO}_{< \ell}$).  

The hypersurface $H_\ell$ splits into $\ell-1$ smooth parts depending on the location of the double real root among all the real roots.   
}
\end{remark}

\begin{definition} By a \textcolor{blue} {closed non-convex polyhedral cone} we mean the union of an arbitrary finite collection of closed polyhedral cones. 
\end{definition} 

\begin{definition} We say that a closed (probably non-convex) polyhedral cone  $\CC\subset \bR_n[x]$ 
contained  in $\OO_{\ge \ell}$ (resp.  $\OO_{< \ell}$) \textcolor{blue} {is inscribed} in $\OO_{\ge \ell}$ 
(resp.  $\OO_{< \ell}$) if one cannot freely deform any of its vertices and still obtain a cone contained in the domain.  

Analogously, we say that  a closed (probably non-convex) polyhedral cone $\CC\subset \bR_n[x]$ containing  
$\OO_{\ge \ell}$ (resp.  $\OO_{< \ell}$) \textcolor{blue} {is circumscribed} if one can not freely deform any 
of its vertices and still obtain a cone containing the domain.
\end{definition} 





The main question we consider below is as follows. 

\begin{problem}\label{prob:main} Find interesting examples of inscribed and circumscribed polyhedral cones for the domains $\OO_{< \ell}$ and $\OO_{\ge \ell}$. \end{problem} 

Observe that any such cone  provides a generalization of either Hutchinson theorem or Newton inequalities respectively.   
Below we present some partial, but non-trivial results related to Problem~\ref{prob:main}. We are currently looking for a more conceptual approach to this question. 

\medskip
\noindent
{\bf Acknowledgements.} 
 The research of the second named author was supported by the Swedish Research Council grant 2021-04900. He is grateful to Robin Stoll for topological consultations. All three authors  want to acknowledge the importance of their long-term contacts with the Professor V.~Kostov of Universit\'e  C\^ote d'Azur to whom this note is dedicated.  We highly appreciate the relevant comments of the anonymous referee which allowed us to improve the overall exposition and correct a number of small mistakes.

\section{Preliminary results} 

\begin{lemma}\label{lm:compl} {\rm The complement  $\Omega^+_n:=\bR^+_n[x]\setminus \D^+_n$ is the union of  $\left[\frac{n}{2}\right]+1$ open contractible components  $\Omega^+_{n, \ell},\; 0\le \ell\le n,\; \ell\equiv n \mod 2$, where  $\Omega^+_{n, \ell}$ is the set of  all monic degree $n$ polynomials with positive coefficients having $\ell$ simple real (negative) roots. }
\end{lemma} 

\begin{proof} Observe that $\Omega^+_n:=\bR^+_n[x]\setminus \D^+_n$ is an open set and therefore its connected components are open as well. Obviously the number of real and simple roots is constant within each such connected component. Let us show that the set $\Omega^+_{n, \ell}$ consisting of all monic degree $n$ polynomials with positive coefficients and exactly $\ell$ distinct real roots where $\ell \le n$ and $\ell\equiv n \mod 2$ is open, non-empty, and contractible. This fact will settle our lemma. 

To prove non-emptiness observe that the polynomial $$p_{n,\ell}(x)=(x+1)(x+2)\cdots (x+\ell)(x^2+1)^{\frac{n-\ell}{2}}$$ belongs to $\Omega^+_{n, \ell}$. To prove contractibility we will show that any compact subset of  $\Omega^+_{n, \ell}$ is contractible. (The fact that contractibility of any compact subset in an open subset  $S\subseteq \bR^n$ implies contractibility of  $S$ follows from Whitehead's theorem  together with Corollary IV.5.5 
of  \cite{LuWe}). To do this we first observe the contractibility of the subset $\widetilde{\Omega^+_{n, \ell}}\subset \Omega^+_{n, \ell}$ consisting of all polynomials of the form 
$$p(x)=(x+x_1)(x+x_2)\dots (x+x_\ell)(x^2+p_1x+q_1)\dots (x^2+p_\kappa x + q_\kappa)$$
where $\kappa=\frac{n-\ell}{2}$, $x_1>0, x_2>0,  \dots, x_\ell>0$,  $p_1>0, p_2>0,  \dots, p_\kappa>0$, and $q_1> \frac{p_1^2}{4}, \dots, q_\kappa > \frac {p_\kappa^2}{4}$. 

Indeed, we can contract $\widetilde{\Omega^+_{n, \ell}}$ to the polynomial $p_{n,\ell}(x)\in \widetilde{\Omega^+_{n, \ell}}$ given by 
$$p_{n,\ell}(x)=(x+1)^\ell(x^2+x+1)^\kappa. 
$$
To do this we (linearly) deform each factor $u_i(x)=(x+x_i)$ into $(x+1)$ by using the family $u_i(x,t)={t(x+1)+(1-t)u_i(x)},\; t\in [0,1].$ Analogously, we can linearly deform each factor 
$v_i(x)=x^2+p_i x+q_i$ into $x^2+x+1$ by using the family $v_i(x,t)={t(x^2+x+1)+(1-t)v_i(x)},\; t\in [0,1].$ One can easily check that the latter deformation preserves the condition $p>0$ and $q>\frac{p^2}{4}$. 

Now given any compact subset of $S\subset \Omega^+_{n, \ell}$, we can move it within  $\Omega^+_{n, \ell}$ into $\widetilde{\Omega^+_{n, \ell}}$ where it can be contracted to $p_{n,\ell}(x)$. Indeed, for any polynomial $p(x)$ consider its deformation $p_t(x)=p(x+t)$ where $t\in [0,+\infty)$.  One can check that for any $p(x)\in \Omega^+_{n, \ell}$, there exists $t_p$ such that for all $t>t_p$ the polynomial $p_t(x)$ belongs to $\widetilde{\Omega^+_{n, \ell}}$. One can easily see that $t_p$ depends continuously on $p$. Thus we can move any compact set $S\subset \Omega^+_{n, \ell}$ in  $\widetilde{\Omega^+_{n, \ell}}$  and contract it. This fact along with a number of similar statements can be found in \cite{Ko}.
\end{proof} 

\begin{corollary}\label{cor:compl} {\rm The complement  $\widehat\Omega^+_n:=\widehat \bR^+_n[x]\setminus \widehat\D^+_n$ is the union of  of $\left[\frac{n}{2}\right]+1$ open contractible components  $\widehat \Omega^+_{n, \ell},\; 0\le \ell\le n,\; \ell\equiv n \mod 2$, where  $\widehat \Omega^+_{n, \ell}$ is the set of  all monic degree $n$ polynomials with positive coefficients and constant term $1$ having $\ell$ simple real (negative) roots. } 
\end{corollary} 

\begin{proof} Use the quasihomogeneous action of $\bR^+$ on $\bR_n[x]$. 
\end{proof} 

\begin{remark} Observe that any real polynomial with positive coefficients has only negative real roots. 
\end{remark}

\begin{lemma}\label{lm:simp} {\rm (i) In the above notation, the map $\text{Log} |\cdot |$ sends $\widehat \bR^+_{\bar \si, \bar \epsilon}$ to the affine cone given by affine  inequalities 
$$2\al_j-\al_{j-1}-\al_{j+1}\lessgtr_{\bar \si} \log \epsilon_j,\; j=1, \dots, n-1,$$
where the inequality sign of the $j$-th inequality is given by $\si_j$. (Here $\al_0=\al_{n}=0$).

\noindent
(ii) the affine cone $\text{Log} \left(\widehat \bR^+_{\bar \si, \bar \epsilon}\right)$ is an (affine) orthant in $\bR^{n-1}$ in appropriate coordinates;

\noindent
 (iii) For two $(n-1)$-tuples $\bar \epsilon^{(1)}$ and $\bar \epsilon^{(2)}$ of positive numbers, the set  $\widehat \bR^+_{\bar \si, \bar \epsilon^{(1)}}$ is contained in the set $\widehat \bR^+_{\bar \si, \bar \epsilon^{(2)}}$ if and only if $\bar \epsilon^{(1)} \lessgtr_{\bar \sigma} \bar \epsilon^{(2)}$, which means that $\epsilon_j^{(1)} \lessgtr_{\bar \sigma} \epsilon_j^{(2)},\; j=1,\dots, n-1$ and the sign of the inequality is determined by $\si_j$. }
\end{lemma} 

\begin{proof} Item \rm {(i)} follows immediately by taking the logarithm of the inequalities defining  $\widehat \bR^+_{\bar \si, \bar \epsilon}$. 

To settle \rm {(ii)}, introduce $\kappa_j:=\log q_j=2\al_j-\al_{j-1}-\al_{j+1},\; j=1,\dots, n-1.$ We show that $(\kappa_1,\dots, \kappa_{n-1})$ is a coordinate system in $\bR^{n-1}$. Indeed, one has the following relation 
$$\begin{pmatrix}\kappa_1\\ \kappa_2\\\kappa_3\\ \vdots\\ \kappa_{n-1} \end{pmatrix}=\begin{pmatrix}2&-1&0&0&\dots &0\\ -1&2&-1&0&\dots&0\\  0&-1&2&-1&\dots&0 \\
\vdots&\vdots&\vdots&\vdots&\ddots&\vdots \\ 0&0&\dots&0&-1&2   \end{pmatrix}   \begin{pmatrix}\alpha_1\\ \alpha_2 \\ \alpha_3\\\vdots\\ \alpha_{n-1} \end{pmatrix}.
$$
The determinant of the $(n-1)\times (n-1)$-matrix in the right-hand side of the latter equation equals $n$ which can be easily proved by induction which implies item {(ii)} since $\text{Log} \left(\widehat \bR^+_{\bar \si, \bar \epsilon}\right)$ is given by the system of inequalities $\kappa_j\lessgtr_{\bar \si} \log \epsilon_j,\; j=1,\dots, n-1$.  

Item \rm{(iii)} follows immediately from comparison of shifted coordinate orthants. 
\end{proof}

The following lemma is straightforward.\\

\noindent
\begin{lemma}\label{lem:1}  {\rm Let $P(x)=\sum_{k=0}^{n} a_k x^k  \in \mathbb{R}[x].$
\begin{enumerate}
\item If $\mathfrak {P}(x)=a P(bx)$ for some non-zero real values of $a$  and $b,$ then
\begin{equation}
\label{a6}
q_j(\mathfrak{P})=q_j(P), \ \ j=1,2, \ldots, n-1
\end{equation}
\item If $\widetilde{P}(x)=x^n P\left(\frac{1}{x}\right),$ then
\begin{equation}
\label{a7}
q_j(\widetilde {P})=q_{n-j}(P), \ \ j=1,2, \ldots, n-1. 
\end{equation}
\end{enumerate}
}
\end{lemma}

\section{Main  results}

\subsection{Previously known facts}

In \cite{KV} the first and the third authors proved the following statements.\\

\noindent
{\bf Theorem A.} For $P(x)=\sum_{k=0}^{2m} a_k x^k  \in \mathbb{R}_{+}[x],$ if the inequalities 

\begin{equation}
\label{a2}
q_{2k+1}<\frac{1}{\cos^2 \left( \frac{\pi}{m+2} \right)}
\end{equation}
hold for all $k=0,1,\ldots , m,$ then $P(x)>0$ for each real value of $x,$  i.e. $P(x)\in \OO_{<2}^+$.\\

\noindent
{\bf Theorem B.} For $P(x)=\sum_{k=0}^{2m+1} a_k x^k  \in \mathbb{R}_{+}[x],$ if the inequalities 

\begin{equation}
\label{a3}
q_{2k}<\frac{4k^2-1}{4k^2} \frac{1}{\cos^2 \left( \frac{\pi}{m+2} \right)}
\end{equation}
hold for all $k=1,2,\ldots , m,$ then $P(x)$ has exactly one real (negative) root (counting its multiplicity), i.e. $P(x)\in \OO_{<3}^+$. \\

\noindent
{\bf Theorem C.} The constants $\frac{1}{\cos^2 \left( \frac{\pi}{m+2} \right)}$ in Theorem A and $\frac{4k^2-1}{4k^2} \frac{1}{\cos^2 \left( \frac{\pi}{m+2} \right)}$ in Theorem B are sharp for every $m\in \mathbb{N}.$\

\medskip
\noindent 
\begin{corollary}

\smallskip {\rm 
The assumptions of Theorems A and B can be interpreted as the fact that the logarithmic image of the polynomials under consideration belongs to the polyhedral cone given as follows: 

\noindent
In case of Theorem A: \quad $2\al_{2k+1}-\al_{2k}-\al_{2k+2}<-2\ln (\cos \left(\frac{\pi}{m+2}\right))$, $k=0,1,\ldots , m,$.

\noindent
In case of Theorem B: \quad $2\al_{2k}-\al_{2k-1}-\al_{2k+1}<-2\ln (\cos \left(\frac{\pi}{m+2}\right))+\ln (\frac{4k^2-1}{4k^2})$, $k=1,2, \ldots , m$. 

\medskip
Theorem C guarantees that there exist small deformations of the above polyhedral cones which are not contained in the logarithmic image of 
$ \OO_{<2}^+$ and  $\OO_{<3}^+$ respectively. 
}
\end{corollary} 

\smallskip

Theorems A and B can be generalized to the  case when some coefficients are allowed to be negative. Namely,  the following modifications of these results hold.\\

\noindent
{\bf Theorem D.}  {\rm  For $P(x)=\sum_{k=0}^{2m} a_k x^k  \in \mathbb{R}[x],$ assume that $a_{2k}>0, \ \ k=0, 1, \ldots, m.$ If the inequalities 

\begin{equation}
\label{a4}
q_{2k+1}<\frac{1}{\cos^2 \left( \frac{\pi}{m+2} \right)}
\end{equation}
hold for all $k=0,1,\ldots , m,$ then $P(x)>0$ for each real value of $x,$ i.e. $P(x)\in \OO_{<2}$.\\}

\noindent
{\bf Theorem E.} Given $P(x)=\sum_{k=0}^{2m+1} a_k x^k  \in \mathbb{R}[x],$ assume that $a_{2k+1}>0, \ \ k=0,1,\ldots , m.$ If the inequalities 

\begin{equation}
\label{a5}
q_{2k}<\frac{4k^2-1}{4k^2} \frac{1}{\cos^2 \left( \frac{\pi}{m+2} \right)}
\end{equation}
hold for all $k=1,2,\ldots , m,$ then $P(x)$ has exactly one real root (counting its multiplicity), i.e. $P(x)\in \OO_{<3}$. \\

The next result  follows from Theorem E and the second statement of Lemma~\ref{lem:1}.\\

\noindent
{\bf Theorem F.} Given $P(x)=\sum_{k=0}^{2m+1} a_k x^k  \in \mathbb{R}[x],$ assume that $a_{2k+1}>0, \ \ k=0,1,\ldots , m.$ If the inequalities 

\begin{equation}
\label{a8}
q_{2m+1-2k}<\frac{4k^2-1}{4k^2} \frac{1}{\cos^2 \left( \frac{\pi}{m+2} \right)}
\end{equation}
hold for all $k=1,2,\ldots , m,$ then $P(x)$ has exactly one real root (counting its multiplicity), i.e. $P(x)\in \OO_{<3}$. \\

\medskip
\subsection{New results}
We will use the following notation.\\

For each real polynomial $P(x)$, we will denote by $\sharp_r(P)$ the number of its real roots (counting their multiplicities).\\

\begin{theorem}\label{th:1}
  {\rm Assume that $P(x)=\sum_{k=0}^{n} a_k x^k \in \mathbb{R}^{+}_n[x], $  where $n\geq 4.$ If $\sharp_r(P)\geq n-2,$ then 
 \begin{equation}\label{bb1}
 q_1+q_{n-1}\ge 4\frac{n^2-3n}{(n-2)^2}.
 \end{equation}
In particular, 
\begin{equation}
\label{b1}
\mbox{either} \ \ q_1 \geq \frac{2n^2-6n}{(n-2)^2} \  \   \mbox{or}   \  \ q_{n-1} \geq\frac{2n^2-6n}{(n-2)^2}.
\end{equation}
}
\end{theorem}

The polyhedral domain given by inequalities~\eqref{b1} contains $\OO_{\ge n-2}^+$. Below we will split Theorem 5 into Theorems~\ref{th:3} and \ref{th:4} covering  the cases of even and odd degrees respectively. 

\medskip
As an immediate consequence of Theorem~\ref{th:1} we get the following 
corollary. 

\begin{corollary}\label{cor:1}
{\rm  If 
\begin{equation}
\label{c1}
\max(q_1,q_{n-1}) < \frac{2n^2-6n}{(n-2)^2} ,
\end{equation}\\
 then $$\sharp_r(P)\leq n-4.$$
 }
 \end{corollary}
 
 The polyhedral domain given by these inequalities is contained in $\OO_{< n-2}^+$.

\begin{corollary}\label{cor:2}
{\rm 
 Assume that $P_n(x)=\sum_{k=0}^{n} a_k x^k\in \mathbb{R}_{+}[x], $  where 
 $n\geq 4.$ If for some  $ m=2, 3, \ldots, n-2$ and for some $j=1,2,\ldots, n-m-1,$  
 the following two estimations are valid 
\begin{equation}
\label{b2}
q_j< \frac{(m-1)(m+1)}{m^2} \cdot \frac{n-j+1}{n-j} \cdot \frac{j+1}{j},
\end{equation}
  and 
\begin{equation}
\label{b3}
 q_{m+j}< \frac{(m-1)(m+1)}{m^2} \cdot \frac{n-m-j+1}{n-m-j} \cdot \frac{m+j+1}{m+j},
\end{equation}
 then

 $$\sharp_r(P_n)\leq n-4.$$\\
 }
 \end{corollary}

For the proof of Corollary~\ref{cor:2} see \S~\ref{sec:proofs}.

\begin{theorem}\label{th:2}
{\rm 
The estimate in Theorem~\ref{th:1} is sharp.\\
}
\end{theorem}

{\it Proof.} The following example  
\begin{equation}
\label{a14}
P_{n}(x)=(x+1)^{n-2}\left(x^2+\frac{n-4}{n-2}x+1\right) 
\end{equation}
with $\sharp _r(P_n)=n-2$ and $q_1=q_{n-1}=\frac{2n(n-3)}{(n-2)^2}$  
settles  Theorem~\ref{th:2}. \qed \\

\begin{proposition}\label{pr:1}
{\rm 
 The constants in Corollary~\ref{cor:2} are sharp. }
 \end{proposition}
 
For the proof of Proposition~\ref{pr:1} see \S~\ref{sec:proofs}.

\medskip
Theorem~\ref{th:1} is equivalent to the following 2 statements whose rather lengthy proofs  can be found in \S~\ref{sec:proofs}. 

\begin{theorem}\label{th:3} 
{\rm 
If 

\begin{equation}
\label{a10}
P_{2n}= \left(x^2+2tx+\frac{1}{b_1b_2\cdots b_{n-1}}\right) \prod_{j=1}^{n-1} (x^2+2a_jx+b_j),
\end{equation}
where $a_j>0, \ \  b_j>0, \ \ a_j^2 \geq b_j, \ \ j=1,2,\ldots, n-1,$ then 

\begin{equation} \label{aa1}
q_1+q_{2n-1} \ge 2\cdot\frac{2n^2-3n}{(n-1)^2}.
\end{equation}

In particular, 
\begin{equation}
\label{a11}
\mbox{either} \ \ q_1 \geq \frac{2n^2-3n}{(n-1)^2} \  \   \mbox{or}   \  \ q_{2n-1} \geq \frac{2n^2-3n}{(n-1)^2}. 
\end{equation}
}
\end{theorem}

\begin{theorem}\label{th:4} 
{\rm 
If 

\begin{equation}
\label{a12}
P_{2n+1}=(x+c) \left(x^2+2tx+\frac{1}{b_1b_2\cdots b_{n-1}c}\right) \prod_{j=1}^{n-1} (x^2+2a_jx+b_j),
\end{equation}
where $c>0, \ \ a_j>0, \ \  b_j>0, \ \ a_j^2 \geq b_j, \ \ j=1,2,\ldots, n-1,$ then 

\begin{equation}\label{aa13}
q_1+q_{2n}\ge 8\cdot \frac{2n^2-n-1}{(2n-1)^2}.
\end{equation}

In particular, 
\begin{equation}
\label{a13}
\mbox{either} \ \ q_1 \geq  4\cdot \frac{2n^2-n-1}{(2n-1)^2} \  \   \mbox{or}   \  \ q_{2n} \geq 4 \cdot \frac{2n^2-n-1}{(2n-1)^2}.
\end{equation}
}
\end{theorem}

\begin{notation}
We will use the following standard notation for the symmetric functions. For 
$\alpha=(\alpha_1, \alpha_2, \ldots, \alpha_{n-1}),$ 
set
\begin{equation}
\label{a14}
\sigma_1(\alpha)=\alpha_1+\alpha_2+ \cdots+ \alpha_{n-1};
\end{equation}

\begin{equation}
\label{a15}
\sigma_2(\alpha)=\prod_{1\leq i<j\leq n-1}\alpha_i \alpha_j;
\end{equation}

\begin{equation}
\label{a16}
S_2(\alpha)=\alpha_1^2+\alpha_2^2+ \cdots+ \alpha_{n-1}^2.
\end{equation}
\end{notation}

\medskip The next claim is standard. 

\begin{lemma}\label{lem:2} {\rm In the above notation, 
 $\sigma_1^2(\alpha)=S_2(\alpha)+2\sigma_2(\alpha)$ and $2\sigma_2(\alpha)\leq(n-2) S_2(\alpha).$} 
\end{lemma}

\medskip
To formulate our next result,  notice that  for $P(x)=\sum_{k=0}^{n} a_k x^k \in 
\mathbb{R}_{+}[x] $, one obtains 
\begin{align}
\label{q1111}
&a_k  = a_1\Big(\frac{a_1}{a_0} \Big)^{k-1} \frac{1}{q_1^{k-1}q_2^{k-2}
\cdot \ldots \cdot q_{k-2}^2 q_{k-1}}, \quad 
\forall k,  2 \leq k \leq n - 1.
\end{align}

\medskip
The following proposition is an analog of the  Hutchinson theorem.

\begin{proposition}\label{prop:anya} {\rm 
Let $P(x)=\sum_{k=0}^{n} a_k x^k \in 
\mathbb{R}_{+}[x] $ and suppose that $q_k(P) \geq 1$ for all
$k,  1 \leq k \leq n - 1.$ If for some $j,  1 \leq j \leq n - 1,$ we
have $q_j(P) \geq 4, $ then there exists a point $x_j \in 
\left(- \frac{a_j}{a_{j+1}}, - \frac{a_{j-1}}{a_{j}}\right)$
such that $(-1)^j P(x_j) \geq 0.$}
\end{proposition}

For the proof of Proposition~\ref{prop:anya} see \S~\ref{sec:proofs}.

\section{Proofs}\label{sec:proofs}

{\it Proof of Theorem~\ref{th:3}.}
We will deal with the sequences:

$$a=(a_1,a_2, \ldots, a_{n-1}), \ \ b=(b_1,b_2, \ldots, b_{n-1}), $$

\begin{equation}
\label{a17}
\frac{1}{b}=\left(\frac{1}{b_1},\frac{1}{b_2}, \ldots, \frac{1}{b_{n-1}}\right), \ \ \frac{a}{b}=\left(\frac{a_1}{b_1},\frac{a_2}{b_2}, \ldots, \frac{a_{n-1}}{b_{n-1}}\right). 
\end{equation}

$$ab^{n-2}=\left(a_1 b_2 b_3 \cdots b_{n-1}, \ a_2 b_1 b_3 \cdots b_{n-1}, \ldots, \  a_{n-1}b_1 b_2 \cdots b_{n-2}\right).$$

Hence $\sigma _1(\frac{a}{b})b_1b_2\cdots b_{n-1}=\sigma _1(ab^{n-2})$.   We have 

$$ P_{2n}= \left(x^2+2tx+\frac{1}{b_1b_2\cdots b_{n-1}}\right) \prod_{j=1}^{n-1} (x^2+2a_jx+b_j)$$

$$= x^{2n}+2x^{2n-1}(\sigma_1(a)+t)$$ 
$$ + x^{2n-2} \left(4t \sigma_1(a)+4 \sigma_2(a) +\sigma_1(b)+\frac{1}{b_1b_2\cdots b_{n-1}}\right) $$

$$+\cdots + x^2 \left(4t \sigma_1(ab^{n-2})+4\sigma_2 \left(\frac{a}{b} \right)+\sigma_1\left(\frac{1}{b} \right)+b_1 b_2\cdots b_{n-1}\right)$$ 
$$ + 2x\left(\sigma_1\left(\frac{a}{b} \right)+b_1 b_2\cdots b_{n-1}t \right)+1.$$

Therefore,

\begin{equation}
\label{a18}
q_{2n-1}=\frac{(\sigma_1(a)+t)^2}{t \sigma_1(a)+ \sigma_2(a) +\frac{1}{4}\cdot \sigma_1(b)+\frac{1}{4}\frac{1}{b_1b_2\cdots b_{n-1}}}
\end{equation}

and

\begin{equation}
\label{a19}
q_{1}=\frac{\left(\sigma_1\left(\frac{a}{b} \right)+b_1 b_2\cdots b_{n-1}t \right)^2}{t \sigma_1(ab^{n-2})+ \sigma_2 \left(\frac{a}{b} \right) +\frac{1}{4}\cdot \sigma_1\left(\frac{1}{b} \right)+\frac{1}{4}b_1b_2\cdots b_{n-1}}.
\end{equation}

\medskip
Identities (\ref{a18}) and (\ref{a19}) can be rewritten in the form:

$$t^2+(2-q_{2n-1})\sigma_1(a)t+\sigma_1^2(a) -q_{2n-1} \sigma_2(a)$$

\begin{equation}
\label{a20}
- \frac{1}{4}q_{2n-1}\cdot \sigma_1(b)-\frac{q_{2n-1}}{4b_1b_2\cdots b_{n-1}}=0
\end{equation}

and

$$(b_1b_2\cdots b_{n-1})^2 t^2+(2-q_1)\sigma_1(ab^{n-2})t+\sigma_1^2\left(\frac{a}{b}\right) $$

\begin{equation}
\label{a21}
 - q_1 \sigma_2\left(\frac{a}{b}\right) -\frac{1}{4} q_1\cdot \sigma_1 \left( \frac{1}{b}\right)-\frac{q_1}{4} \cdot b_1b_2\cdots b_{n-1}=0.
\end{equation}

We have two quadratic equations with respect to the variable $t$. Given that these equations have real solutions, i.e. that  their discriminants are non-negative, we obtain the following two inequalities: 

\begin{equation}
\label{a22}
(2-q_{2n-1})^2\sigma_1^2 (a)-4\sigma_1^2 (a)+4q_{2n-1} \sigma_2(a)+ q_{2n-1}\cdot \sigma_1(b)+\frac{q_{2n-1}}{b_1b_2\cdots b_{n-1}}\geq 0
\end{equation}

and

$$(2-q_1)^2 \sigma_1^2(ab^{n-2})-4(b_1b_2\cdots b_{n-1})^2\sigma_1^2\left(\frac{a}{b}\right)$$

\begin{equation}
\label{a23}
 + q_1\cdot (b_1b_2\cdots b_{n-1})^2
\left(4\sigma_2\left(\frac{a}{b}\right) + \sigma_1 \left( \frac{1}{b}\right)+b_1b_2\cdots b_{n-1}\right)\geq 0.
\end{equation}

The inequality (\ref{a22}) is equivalent to the estimate 

\begin{equation}
\label{a24}
q_{2n-1} \geq 4-\frac{4 \sigma_2(a)+\sigma_1(b)+\frac{1}{b_1b_2\cdots b_{n-1}}}{\sigma_1^2(a)}.
\end{equation}

Note that

\begin{equation}
\label{a25}
(b_1b_2\cdots b_{n-1})^2\sigma_1^2\left(\frac{a}{b}\right)=\sigma_1^2(ab^{n-2}),
\end{equation}

and

\begin{equation}
\label{a26}
(b_1b_2\cdots b_{n-1})^2\sigma_2\left(\frac{a}{b}\right)=\sigma_2(ab^{n-2}).
\end{equation}

So, we can rewrite the inequality (\ref{a23}) in the form

\begin{equation}
\label{a27}
q_1 \geq 4-\frac{4 \sigma_2(ab^{n-2})+(b_1b_2\cdots b_{n-1})^2\sigma_1\left(\frac{1}{b}\right)+(b_1b_2\cdots b_{n-1})^3}{\sigma_1^2(ab^{n-2})}. 
\end{equation}

Given that $a_j^2\geq b_j, \ \ j=1, 2, \ldots, n-1,$ the  following inequalities hold:

\begin{equation}
\label{a28}
\sigma_1(b)\leq S_2(a)
\end{equation}

and

\begin{equation}
\label{a29}
(b_1b_2\cdots b_{n-1})^2\sigma_1\left(\frac{1}{b}\right) \leq S_2(ab^{n-2}). 
\end{equation}

Thus, it follows from (\ref{a24}) and the first statement of Lemma~\ref{lem:2} that

\begin{equation}
\label{a30}
q_{2n-1} \geq 4-\frac{4 \sigma_2(a)+S_2(a)+\frac{1}{b_1b_2\cdots b_{n-1}}}{S_2(a)+2\sigma_2(a)} \geq 3-\frac{2 \sigma_2(a)+\frac{1}{b_1b_2\cdots b_{n-1}}}{S_2(a)+2\sigma_2(a)}
\end{equation}

and from (\ref{a27}) that

$$
q_1 \geq 4-\frac{4 \sigma_2(ab^{n-2})+S_2(ab^{n-2})+(b_1b_2\cdots b_{n-1})^3}{S_2(ab^{n-2})+2\sigma_2(ab^{n-2})}$$

\begin{equation}
\label{a31}
\geq 3-\frac{2 \sigma_2(ab^{n-2})+(b_1b_2\cdots b_{n-1})^3}{S_2(ab^{n-2})+2\sigma_2(ab^{n-2})}. 
\end{equation}

Inequalities (\ref{a30}) and (\ref{a31}) imply that the statement of Theorem~\ref{th:3} would follow from the fact that one of the two inequalities below is valid: either

\begin{equation}
\label{a32}
\frac{2 \sigma_2(a)+\frac{1}{b_1b_2\cdots b_{n-1}}}{S_2(a)+2\sigma_2(a)}\leq \frac{n^2-3n+3}{(n-1)^2}
\end{equation}

or

\begin{equation}
\label{a33}
\frac{2 \sigma_2(ab^{n-2})+(b_1b_2\cdots b_{n-1})^3}{S_2(ab^{n-2})+2\sigma_2(ab^{n-2})}\leq \frac{n^2-3n+3}{(n-1)^2}.
\end{equation}

\medskip
The inequality (\ref{a32}) is equivalent to 
\begin{equation}
\label{a34}
1 \leq  \frac{b_1b_2\cdots b_{n-1}}{(n-1)^2}\left( (n^2-3n+3)S_2(a)-(n-2) \cdot 2 \sigma_2(a)\right),
\end{equation}

while (\ref{a33}) is equivalent to the inequality 

\begin{equation}
\label{a35}
 1\leq \frac{1}{(n-1)^2 (b_1b_2\cdots b_{n-1})^3}\left((n^2-3n+3)S_2(ab^{n-2})-(n-2)\cdot 2 \sigma_2(ab^{n-2})\right). 
\end{equation}

\medskip
Note that applying  the second statement of Lemma~\ref{lem:2} and after that inequalities $a_j^2\geq b_j, \ \ j=1,2,\ldots, n-1,$ we get

$$\frac{b_1b_2\cdots b_{n-1}}{(n-1)^2}\left((n^2-3n+3)S_2(a)-(n-2) \cdot 2 \sigma_2(a)\right) $$

$$ \geq \frac{b_1b_2\cdots b_{n-1}}{(n-1)^2}\left((n^2-3n+3)S_2(a)-(n-2)^2 S_2(a)\right)$$

$$ = \frac{b_1b_2\cdots b_{n-1}}{(n-1)} S_2(a)\geq \frac{b_1b_2\cdots b_{n-1}\sigma_1(b)}{(n-1)}$$

\begin{equation}
\label{a36}
  =\frac{1}{n-1} \left( b_1^2 b_2\cdots b_{n-1}+b_1b_2^2 \cdots b_{n-1}+b_1b_2\cdots b_{n-1}^2 \right),
\end{equation}

and

$$ \frac{1}{(n-1)^2 (b_1b_2\cdots b_{n-1})^3}\left((n^2-3n+3)S_2(ab^{n-2})-(n-2)\cdot 2 \sigma_2(ab^{n-2})\right)$$

$$ \geq \frac{1}{(n-1)^2 (b_1b_2\cdots b_{n-1})^3}\left((n^2-3n+3)S_2(ab^{n-2})-(n-2)^2S_2(ab^{n-2})\right) $$ 

$$=\frac{1}{(n-1) (b_1b_2\cdots b_{n-1})^3}S_2(ab^{n-2}) $$

$$ \geq \frac{\left( b_1 b_2^2\cdots b_{n-1}^2+b_1^2 b_2 \cdots b_{n-1}^2+b_1^2 b_2^2\cdots b_{n-1}^2 \right)}{(n-1) (b_1b_2\cdots b_{n-1})^3}$$

\begin{equation}
\label{a37}
  =\frac{1}{n-1} \left( b_1^{-2} b_2^{-1}\cdots b_{n-1}^{-1}+b_1^{-1}b_2^{-2} \cdots b_{n-1}^{-1}+b_1^{-1}b_2^{-1}\cdots b_{n-1}^{-2 }\right).
\end{equation}

\medskip

Denote by 
\begin{equation}
\label{a38}
 B_1=:\frac{1}{n-1} \left( b_1^2 b_2\cdots b_{n-1}+b_1b_2^2 \cdots b_{n-1}+b_1b_2\cdots b_{n-1}^2 \right) 
\end{equation}
and by
\begin{equation}
\label{a39}
 B_2=:\frac{1}{n-1} \left( (b_1^2 b_2\cdots b_{n-1})^{-1}+(b_1b_2^2 \cdots b_{n-1})^{-1}+(b_1b_2\cdots b_{n-1}^2)^{-1} \right).
\end{equation}

Theorem~\ref{th:3} will be proved if we show that $B_1+B_2 \geq 2.$ Indeed, 
since for each $x>0$, the inequality
$x+\frac{1}{x}\geq 2$
is valid, we get

$$B_1+B_2\geq \frac{1}{n-1}\cdot 2 (n-1)=2.$$

In particular, at least one of the numbers $B_1$ and $B_2$ must be bigger than $1.$

Theorem~\ref{th:3} follows. \qed\\

{\it Proof of Theorem~\ref{th:4}.}  (The logic of this proof is very similar to that of Theorem~\ref{th:3}). 

We have

$$ P_{2n+1}=(x+c) \left(x^2+2tx+\frac{1}{b_1b_2\cdots b_{n-1}c}\right) \prod_{j=1}^{n-1} (x^2+2a_jx+b_j)$$

$$ = x^{2n+1}+2x^{2n}(\sigma_1(a)+\frac{c}{2}+t)+$$ 

$$x^{2n-1} \left(4t \left(\sigma_1(a)+\frac{c}{2}\right)+4 \sigma_2(a)+2c\sigma_1(a) +\sigma_1(b)+\frac{1}{b_1b_2\cdots b_{n-1}c}\right)+\cdots+$$

$$x^2 \left(4t \left( c\sigma_1(ab^{n-2})+\frac{b_1 b_2\ldots b_{n-1}}{2} \right)+4\sigma_2 \left(\frac{a}{b} \right)+\frac{2}{c} \sigma_1 \left(\frac{a}{b}\right)+\sigma_1\left(\frac{1}{b} \right)+c b_1 b_2\cdots b_{n-1}\right)$$ $$ + 2x\left(\sigma_1\left(\frac{a}{b} \right)+\frac{1}{2c}+c b_1 b_2\cdots b_{n-1}t \right)+1.$$

Therefore,

\begin{equation}
\label{a40}
q_{2n}=\frac{(\sigma_1(a)+\frac{c}{2}+t)^2}{t (\sigma_1(a)+\frac{c}{2})+ \sigma_2(a) +\frac{c}{2}\sigma_1(a)+\frac{1}{4}\cdot \sigma_1(b)+\frac{1}{4}\frac{1}{c b_1b_2\cdots b_{n-1}}}
\end{equation}

and

$$q_{1}=$$

\begin{equation}
\label{a41}
\frac{\left(\sigma_1\left(\frac{a}{b} \right)+\frac{1}{2c}+c b_1 b_2\cdots b_{n-1}t \right)^2}{t \left(c\sigma_1(ab^{n-2})+\frac{b_1 b_2\ldots b_{n-1}}{2}\right)+ \sigma_2 \left(\frac{a}{b} \right)+\frac{\sigma_1\left(\frac{a}{b}\right)}{2c} +\frac{1}{4}\cdot \sigma_1\left(\frac{1}{b} \right)+\frac{c}{4}b_1b_2\cdots b_{n-1}}.
\end{equation}

The identities (\ref{a40}) and (\ref{a41}) can be rewritten in the form:

$$t^2+(2-q_{2n})\left(\sigma_1(a)+\frac{c}{2}\right)t+\left(\sigma_1(a)+\frac{c}{2}\right)^2-q_{2n} \sigma_2(a)$$

\begin{equation}
\label{a42}
- \frac{c}{2}q_{2n}\sigma_1(a)-\frac{1}{4}q_{2n}\cdot \sigma_1(b)-\frac{q_{2n}}{4c b_1b_2\cdots b_{n-1}}=0
\end{equation}

and

$$(c b_1b_2\cdots b_{n-1})^2 t^2+(2-q_1)\left(c\sigma_1(ab^{n-2})+\frac{b_1b_2\cdots b_{n-1}}{2}\right)t$$ $$ + \left(\sigma_1\left(\frac{a}{b}\right) +\frac{1}{2c}\right)^2-q_1 \sigma_2\left(\frac{a}{b}\right) -\frac{q_1}{2c}\sigma_1\left(\frac{a}{b}\right)$$

\begin{equation}
\label{a43}
-\frac{1}{4} q_1\cdot \sigma_1 \left( \frac{1}{b}\right)-\frac{c q_1}{4} \cdot b_1b_2\cdots b_{n-1}=0.
\end{equation}

We have two quadratic equations with respect to the variable $t$. Given that these equations have real solutions, i.e. that  their discriminants are non-negative, we obtain the inequalities: 

$$(2-q_{2n})^2\left(\sigma_1(a)+\frac{c}{2}\right)^2-4\left(\sigma_1(a)+\frac{c}{2}\right)^2$$
\begin{equation}
\label{a44}
 + 4q_{2n} \sigma_2(a)+2c q_{2n} \sigma_1(a)+ q_{2n}\cdot \sigma_1(b)+\frac{q_{2n}}{cb_1b_2\cdots b_{n-1}}\geq 0
\end{equation}

and

$$(2-q_1)^2\left(c\sigma_1(ab^{n-2})+\frac{b_1 b_2\ldots b_{n-1}}{2}\right)^2-4(c b_1b_2\cdots b_{n-1})^2\left(\sigma_1\left(\frac{a}{b} \right)+\frac{1}{2c}\right)^2$$

\begin{equation}
\label{a45}
+ q_1\cdot ( b_1b_2\cdots b_{n-1})^2
\left(4c^2\sigma_2\left(\frac{a}{b}\right)+2c \sigma_1\left(\frac{a}{b}\right)+ c^2\sigma_1 \left( \frac{1}{b}\right)+c^3 b_1b_2\cdots b_{n-1}\right)\geq 0.
\end{equation}

Inequality (\ref{a44}) is equivalent to the estimate

\begin{equation}
\label{a46}
q_{2n} \geq 4-\frac{4 \sigma_2(a)+2c\sigma_1(a)+\sigma_1(b)+\frac{1}{c b_1b_2\cdots b_{n-1}}}{\left(\sigma_1(a)+\frac{c}{2}\right)^2}.
\end{equation}

Using (\ref{a25}) and (\ref{a26}) we can rewrite (\ref{a45}) in the form

$$q_1 \geq 4-\frac{4c^2 \sigma_2(ab^{n-2})+2(c b_1b_2\cdots b_{n-1})\sigma_1(ab^{n-2})}{\left(c\sigma_1(ab^{n-2})+\frac{b_1 b_2\ldots b_{n-1}}{2}\right)^2}$$

\begin{equation}
\label{a47}
-\frac{(c b_1b_2\cdots b_{n-1})^2\sigma_1\left(\frac{1}{b}\right)+(c b_1b_2\cdots b_{n-1})^3}{\left(c\sigma_1(ab^{n-2})+\frac{b_1 b_2\ldots b_{n-1}}{2}\right)^2}.
\end{equation}

Applying  (\ref{a28}), (\ref{a29}) and Lemma~\ref{lem:2} to (\ref{a46}) we obtain 

$$q_{2n} \geq 4-\frac{4 \sigma_2(a)+2c\sigma_1(a)+S_2(a)+\frac{1}{c b_1b_2\cdots b_{n-1}}}{S_2(a)+2\sigma_2(a)+c\sigma_1(a)+\frac{c^2}{4}} $$

\begin{equation}
\label{a48}
\geq 3-\frac{2 \sigma_2(a)+c \sigma_1(a)+\frac{1}{c b_1b_2\cdots b_{n-1}}-\frac{c^2}{4}}{S_2(a)+2\sigma_2(a)+c\sigma_1(a)+\frac{c^2}{4}}.
\end{equation}\\

Using the same idea we can derive the inequality below from (\ref{a47})\\
$$
q_1 \geq 4-$$ $$\frac{4c^2 \sigma_2(ab^{n-2})+2(c b_1b_2\cdots b_{n-1})\sigma_1(ab^{n-2})}{c^2S_2(ab^{n-2})+2c^2\sigma_2(ab^{n-2})+(c b_1b_2\cdots b_{n-1})\sigma_1(ab^{n-2})+\frac{1}{4}(b_1b_2\cdots b_{n-1})^2}$$ 

$$-\frac{c^2 S_2(ab^{n-2})+(c b_1b_2\cdots b_{n-1})^3}{c^2S_2(ab^{n-2})+2c^2\sigma_2(ab^{n-2})+(c b_1b_2\cdots b_{n-1})\sigma_1(ab^{n-2})+\frac{1}{4}(b_1b_2\cdots b_{n-1})^2}$$

\begin{equation}
\label{a49}
\geq 3-\frac{2 \sigma_2(ab^{n-2})+\frac{b_1b_2\cdots b_{n-1}}{c}\sigma_1(ab^{n-2})+c(b_1b_2\cdots b_{n-1})^3-\frac{(b_1b_2\cdots b_{n-1})^2}{4c^2}}{S_2(ab^{n-2})+2\sigma_2(ab^{n-2})+\frac{b_1b_2\cdots b_{n-1}}{c}\sigma_1(ab^{n-2})+\frac{(b_1b_2\cdots b_{n-1})^2}{4c^2}}.
\end{equation}

It follows from (\ref{a48}) and (\ref{a49}) that the statement of Theorem~\ref{th:4} would follow from the fact that one of the two inequalities below is valid: either

\begin{equation}
\label{a50}
\frac{2 \sigma_2(a)+c \sigma_1(a)+\frac{1}{c b_1b_2\cdots b_{n-1}}-\frac{c^2}{4}}{S_2(a)+2\sigma_2(a)+c\sigma_1(a)+\frac{c^2}{4}}\leq \frac{4n^2-8n+7}{(2n-1)^2}
\end{equation}

or

$$\frac{2 \sigma_2(ab^{n-2})+\frac{b_1b_2\cdots b_{n-1}}{c}\sigma_1(ab^{n-2})+c(b_1b_2\cdots b_{n-1})^3-\frac{(b_1b_2\cdots b_{n-1})^2}{4c^2}}{S_2(ab^{n-2})+2\sigma_2(ab^{n-2})+\frac{b_1b_2\cdots b_{n-1}}{c}\sigma_1(ab^{n-2})+\frac{(b_1b_2\cdots b_{n-1})^2}{4c^2}}$$

\begin{equation}
\label{a51}
\leq\frac{4n^2-8n+7}{(2n-1)^2}.
\end{equation}

The inequality (\ref{a50}) is equivalent to the following one

$$1 \leq  \frac{c b_1b_2\cdots b_{n-1}}{(2n-1)^2}( (4n^2-8n+7)S_2(a)$$

\begin{equation}
\label{a52}
- (4n-6) \cdot 2 \sigma_2(a)-(4n-6)c\sigma_1(a)+(2n^2-3n+2)c^2)
\end{equation}

while (\ref{a51}) is equivalent to the inequality 

$$1\leq \frac{1}{(2n-1)^2 c (b_1b_2\cdots b_{n-1})^3}((4n^2-8n+7)S_2(ab^{n-2})-(4n-6)\cdot 2 \sigma_2(ab^{n-2})$$

\begin{equation}
\label{a53}
 -  (4n-6)\frac{b_1 b_2\cdots b_{n-1}}{c} \sigma_1(ab^{n-2})+(2n^2-3n+2)\frac{(b_1 b_2\cdots b_{n-1})^2}{c^2}).
\end{equation}

Applying  the second statement of Lemma~\ref{lem:2} and
$2c\sigma _1(a)\leq S_2(a)+(n-1)c^2,$
  and after that inequalities 
$a_j^2\geq b_j, \ \ j=1,2,\ldots, n-1,$ to the right-hand part of the above inequalities we will have

$$\frac{c b_1b_2\cdots b_{n-1}}{(2n-1)^2}( (4n^2-8n+7)S_2(a)-(4n-6) \cdot 2 \sigma_2(a)$$

$$ - (4n-6)c\sigma_1(a)+(2n^2-3n+2)c^2))\geq \frac{c b_1b_2\cdots b_{n-1}}{(2n-1)^2}( (4n^2-8n+7)S_2(a)$$

$$ - (4n-6)(n-2)S_2(a)-(2n-3)(S_2(a)+(n-1)c^2)+(2n^2-3n+2)c^2) $$

$$ = \frac{c b_1b_2\cdots b_{n-1}}{(2n-1)^2}((4n-2)S_2(a)+(2n-1)c^2) $$

$$ \geq  \frac{c b_1b_2\cdots b_{n-1}}{(2n-1)^2}((4n-2)\sigma_1(b)+(2n-1)c^2)$$

\begin{equation}
\label{a54}
 \geq  \frac{2}{2n-1} \left(c b_1^2 b_2\cdots b_{n-1}+c b_1b_2^2 \cdots b_{n-1}+
  c b_1b_2\cdots b_{n-1}^2 +\frac{c^3b_1 b_2\cdots b_{n-1}}{2}\right)
\end{equation}

and

$$ \frac{1}{(2n-1)^2 c (b_1b_2\cdots b_{n-1})^3}((4n^2-8n+7)S_2(ab^{n-2})-(4n-6)
\cdot 2 \sigma_2(ab^{n-2})$$

$$- (4n-6)\frac{b_1 b_2\cdots b_{n-1}}{c} \sigma_1(ab^{n-2})+(2n^2-3n+2)
\frac{(b_1 b_2\cdots b_{n-1})^2}{c^2})$$

$$ \geq \frac{1}{(2n-1)^2 c (b_1b_2\cdots b_{n-1})^3}((4n^2-8n+7)
S_2(ab^{n-2})-(4n-6)(n-2)S_2(ab^{n-2}) $$

$$ - (2n-3)\left((n-1)\frac{(b_1 b_2\cdots b_{n-1})^2}{c^2}+ S_2(ab^{n-2})\right)+
(2n^2-3n+2)\frac{(b_1 b_2\cdots b_{n-1})^2}{c^2})$$

$$ = \frac{1}{(2n-1)^2 c (b_1b_2\cdots b_{n-1})^3}\left((4n-2)S_2(ab^{n-2})+
(2n-1)\frac{(b_1 b_2\cdots b_{n-1})^2}{c^2}\right)$$

$$ \geq \frac{2}{2n-1}\left(\frac{\left( b_1 b_2^2\cdots b_{n-1}^2+b_1^2 b_2 b_3^2 \cdots 
b_{n-1}^2+b_1^2 b_2^2\cdots b_{n-2}^2 b_{n-1} \right)}{c (b_1b_2\cdots b_{n-1})^3}+
\frac{1}{2c^3b_1 b_2\cdots b_{n-1}}\right)$$

\begin{eqnarray}
\label{a55} & \nonumber 
  =\frac{2}{2n-1} \left((c b_1^2 b_2\cdots b_{n-1})^{-1}+(c b_1b_2^2 b_3 \cdots b_{n-1})^{-1}+
  (c b_1 b_2\cdots b_{n-2} b_{n-1}^2)^{-1} \right.  \\   &
\left.  +\frac{1}{2c^3b_1 b_2\cdots b_{n-1}}\right).
\end{eqnarray}

Denote by 

$$
 B_1=\frac{2}{2n-1} \left(c b_1^2 b_2\cdots b_{n-1}+c b_1b_2^2 
 \cdots b_{n-1}+c b_1b_2\cdots b_{n-1}^2 +\frac{c^3b_1 b_2\cdots b_{n-1}}{2}\right)
$$
and
$$
 B_2=\frac{2}{2n-1} \left((c b_1^2 b_2\cdots b_{n-1})^{-1}+(c b_1b_2^2 
 \cdots b_{n-1})^{-1}+(c b_1 b_2\cdots b_{n-1}^2)^{-1}+\frac{1}{2c^3b_1 b_2\cdots b_{n-1}}\right).
$$

Theorem~\ref{th:4}  will be proved if we show that $B_1+B_2\geq 1.$ Indeed, 
since for each $x>0$ the inequality
$x+\frac{1}{x}\geq 2$
is valid, we have

$$B_1+B_2\geq \frac{2}{2n-1}\cdot( 2 (n-1)+1)=2.$$

In particular, at least one of the numbers $B_1$ and $B_2$ must be bigger than $1.$

Theorem~\ref{th:4} is proved.    \qed\\

\medskip
Next let us show how to derive Corollary~\ref{cor:2} from Corollary~\ref{cor:1}. 

\medskip
{\it Proof of Corollary~\ref{cor:2}.} 
Consider the polynomial

$$P(x)=a_n x^n+\cdots+a_{m+j+1}x^{m+j+1}+a_{m+j}x^{m+j}+a_{m+j-1}x^{m+j-1}$$
\begin{equation}
\label{a58}
 +  \cdots+a_{j+1}x^{j+1}+a_j x^j+a_{j-1}x^{j-1}+\cdots+a_0. 
\end{equation}\\

Differentiating the polynomial  $P(x)$ $(j-1)$ times we get:

$$P^{(j-1)}(x)=\frac{n!}{(n-j+1)!}a_n x^{n-j+1}+\cdots+\frac{(m+j+1)!}{(m+2)!}a_{m+j+1}x^{m+2}$$ 

$$+\frac{(m+j)!}{(m+1)!}a_{m+j}x^{m+1}+\frac{(m+j-1)!}{m!}a_{m+j-1}x^{m}$$

\begin{equation}
\label{a59}
+  \cdots+\frac{(j+1)!}{2!}a_{j+1}x^2+\frac{j!}{1!} a_j x+(j-1)!a_{j-1}. 
\end{equation}

Consider the polynomial 

$$Q_{n-j+1}(x)=x^{n-j+1} P \left(\frac{1}{x}\right)$$  $$ = \frac{n!}{(n-j+1)!}a_n 
+\cdots+\frac{(m+j+1)!}{(m+2)!}a_{m+j+1}x^{n-m-j-1}$$

$$+\frac{(m+j)!}{(m+1)!}a_{m+j}x^{n-m-j}+\frac{(m+j-1)!}{m!}a_{m+j-1}x^{n-m-j+1}$$

\begin{equation}
\label{a60}
+ \cdots+\frac{(j+1)!}{2!}a_{j+1}x^{n-j-1}+\frac{j!}{1!} a_j x^{n-j}+(j-1)!a_{j-1}x^{n-j+1}. 
\end{equation}\\

Differentiating the above polynomial $(n-m-j-1)$ times we get:

$$Q_{n-j+1}^{(n-m-j-1)}(x)=(n-m-j-1)!\frac{(m+j+1)!}{(m+2)!}a_{m+j+1}$$

$$+\frac{(n-m-j)!}{1!}\cdot \frac{(m+j)!}{(m+1)!}a_{m+j}x+\frac{(n-m-j+1)!}{2!}\cdot \frac{(m+j-1)!}{m!}a_{m+j-1}x^2$$

 $$ + \cdots+\frac{(n-j-1)!}{m!}\cdot \frac{(j+1)!}{2!}a_{j+1}x^m+\frac{(n-j)!}{(m+1)!}\cdot \frac{j!}{1!} a_j x^{m+1}$$

\begin{equation}
\label{a61}
+ \frac{(n-j+1)!}{(m+2)!}\cdot(j-1)!a_{j-1}x^{m+2}. 
\end{equation}\\

Let's evaluate $q_{m+1}$ and $q_1$ for the last polynomial. We have

\begin{equation}
\label{a62}
q_{m+1}\left(Q_{n-j+1}^{(n-m-j-1)}\right)=\frac{2(m+2)}{m+1}\cdot \frac{j}{j+1}\cdot \frac{n-j}{n-j+1}\cdot q_j
\end{equation}

\begin{equation}
\label{a63}
q_1\left(Q_{n-j+1}^{(n-m-j-1)}\right)=\frac{2(m+2)}{m+1}\cdot \frac{m+j}{m+j+1}\cdot \frac{n-m-j}{n-m-j+1}\cdot q_{m+j}. 
\end{equation}\\

Applying the given estimations (\ref{b2}) and (\ref{b3}), we obtain

\begin{equation}
\label{a64}
\max\left(q_1\left(Q_{n-j+1}^{(n-m-j-1)}\right), q_{m+1}\left(Q_{n-j+1}^{(n-m-j-1)}\right)\right)<\frac{2(m+2)(m-1)}{m^2}. 
\end{equation}\\

Note that $\deg Q_{n-j+1}^{(n-m-j-1)}=m+2.$ If we denote by $n =m+2,$ we will obtain (\ref{c1}). Therefore, 
by Corollary~\ref{cor:1}

\begin{equation}
\label{c2}
\sharp_r\left(Q_{n-j+1}^{(n-m-j-1)}\right)\leq m-2. 
\end{equation}

By virtue of Rolle's theorem

\begin{equation}
\label{c3}
\sharp_r(Q_{n-j+1})\leq n-j-3. 
\end{equation}

It is clear that for any polynomial $P(x)=\sum_{k=0}^{n}a_k x^k$, the  following identity holds true 

\begin{equation}
\label{c4}
\sharp_r(P(x))=\sharp_r\left(x^n P\left(\frac{1}{x}\right)\right).
\end{equation}\\

So, by (\ref{a60}) we have

\begin{equation}
\label{c5}
\sharp_r(P^{(j-1)})\leq n-j-3. 
\end{equation}

Applying Rolle's theorem we can conclude that

\begin{equation}
\label{c6}
\sharp_r(P)\leq n-4. 
\end{equation}

Corollary~\ref{cor:2}  is proved.  \qed\\

\begin{proof}[Proof of Proposition~\ref{pr:1}] 
Let
$$f(x)=(x+1)^m(ax^2+bx+c)$$ be a polynomial with positive coefficients,
such that the quadratic factor $ax^2+bx+c$ is irreducible. Firstly we will prove that 
there exists an anti-derivative $P$ of $f$   with positive coefficients  such that 
$P(x) =(x+1)^{m+1} (Ax^2+Bx+C)$, where
the quadratic factor $Ax^2+B x+C$ is irreducible.

Let $F$ be an anti-derivative of $f$ such that $F(0)=0.$ Obviously, all 
coefficients of $F$ are positive. Define the polynomial $P$ as follows
$$
P(x)=F(x)+C, \ \ \mbox{where} \ \ P(-1)=0.
$$
So, 
$$C=-F(-1)=\int_{-1}^{0}(t+1)^m(at^2+bt+c) d t>0.$$\\

Therefore, 

$$P(x)=(x+1)^{m+1}(Ax^2+Bx+C) \in \mathbb{R}_{+}[x],$$
where by virtue of Rolle's theorem the polynomial $Ax^2+Bx+C$ is irreducible,
and $P'(x)=f(x).$\\

Now we can easily prove Proposition~\ref{pr:1}. 
Consider the polynomial 
\begin{eqnarray}
\label{neq1} & P_{m+2}(x) = (x+1)^m (x^2+ \frac{m-2}{m}x +1)= x^{m+2}+
\frac{(m+2)(m-1)}{m} x^{m+1} \\
\nonumber  &
+  \frac{(m+2)(m-1)}{2} x^{m} + \ldots +  \frac{(m+2)(m-1)}{2} x^2 + 
\frac{(m+2)(m-1)}{m} x +1.
\end{eqnarray}
Obviously, for the polynomial $ P_{m+2}$ we have
$$q_1 = q_{m+1} = \frac{2(m+2)(m-1)}{m^2}.  $$
Let us integrate the polynomial $ P_{m+2}$ from (\ref{neq1}) $(n-m-j-1)$
times. We will use the statement in the beginning of the proof to obtain
the resulting polynomial with positive coefficients in the form 
$$  P_{n-j+1} =(x+1)^{n-j-1} (Ax^2 +Bx+C) =\frac{(m+2)!}{(n-j+1)!} x^{n-j+1} 
$$
$$+ \frac{(m+2)(m-1)}{m}\cdot \frac{(m+1)!}{(n-j)!}x^{n-j} +
\frac{(m+2)(m-1)}{2}\cdot \frac{m!}{(n-j-1)!}x^{n-j-1} $$
$$ + \ldots + \frac{(m+2)(m-1)}{2}\cdot \frac{2!}{(n-m-j+1)!}x^{n-m-j+1}  $$
$$+  \frac{(m+2)(m-1)}{m}\cdot \frac{1!}{(n-m-j)!}x^{n-m-j}   + 
 \frac{0!}{(n-m-j-1)!}x^{n-m-j-1}   + \ldots $$
Let us consider the polynomial $\tilde{P}_{n-j+1}(x) =x^{n-j+1} P_{n-j+1}(\frac{1}{x})
= (x+1)^{n-j-1} (C x^2 +Bx+A). $ We will integrate the polynomial $\tilde{P}_{n-j+1}$
$(j-1)$ times. We will use the statement in the beginning of the proof to obtain
the resulting polynomial with positive coefficients in the form 
$$\tilde{P}_{n}(x) = (x+1)^{n-2}(\tilde{A} x^2 + \tilde{B} x + \tilde{C}) = \ldots +
\frac{(m+2)!}{(n-j+1)!}\cdot \frac{0!}{(j-1)!}x^{j-1} $$
$$ + \frac{(m+2)(m-1)}{m}\cdot\frac{(m+1)!}{(n-j)!}\cdot \frac{1!}{j!}x^{j}
+\frac{(m+2)(m-1)}{2}\cdot\frac{m!}{(n-j-1)!}\cdot \frac{2!}{(j+1)!}x^{j+1}   $$
$$ + \ldots +   \frac{(m+2)(m-1)}{2}\cdot\frac{2!}{(n-m-j+1)!}\cdot 
\frac{m!}{(m+j-1)!}x^{m+j-1}    $$
$$ +  \frac{(m+2)(m-1)}{m}\cdot\frac{1!}{(n-m-j)!}\cdot \frac{(m+1)!}{(m+j)!}x^{m+j}  $$
$$ +  \frac{0!}{(n-m-j-1)!}\cdot \frac{(m+2)!}{(m+j+1)!}x^{m+j+1} + \ldots $$
If we evaluate $q_j$ and $q_{m+j}$ for the polynomial $\tilde{P}_{n}$ we obtain
$$q_j =\frac{(m-1)(m+1)}{m^2}\cdot \frac{(n-j+1)}{(n-j)} \cdot \frac{(j+1)}{j} $$
and
$$q_{m+j} =\frac{(m-1)(m+1)}{m^2}\cdot \frac{(n-m-j+1)}{(n-m-j)} \cdot 
\frac{(m+j+1)}{(m+j)}.  $$
Proposition~\ref{pr:1} is proved.
\end{proof}

\begin{proof}[Proof of Proposition~\ref{prop:anya}] Following Hutchinson's idea, 
   for a polynomial $P(x) = \sum_{k=0}^n a_k x^k$  with  positive 
coefficients,  we can assume without loss of generality, that $a_0=a_1=1,$ 
since we can consider a polynomial $T(x) = a_0^{-1} P (a_0 a_1^{-1}x) $  
instead of $P.$ (We use  the fact that such rescaling of $P$  preserves the 
second quotients  $q_k(T) =q_k(P)$ for all $k.$)  For the sake of brevity, we 
further use notation  $q_k$ instead of  $q_k(P).$ Thereafter, we consider 
a polynomial  
\begin{align}
\label{phi}
Q(x) = T(-x) = 1 - x + \sum_{k=2}^n  \frac{ (-1)^k x^k}
{q_2^{k-1} q_3^{k-2} \cdot\ldots\cdot q_{k-1}^2 q_k}
\end{align}
instead of $P$ (see \eqref{q1111} for the formulas for coefficients).

Suppose that   $x \in \left( \frac{a_{j-1}}{a_{j}},  \frac{a_j}{a_{j+1}}\right) = 
(q_1q_2 \cdot \ldots \cdot q_{j-1}, q_1q_2 \cdot \ldots \cdot q_{j-1} q_{j}).$  
Since  $q_k \geq 1$ for all
$k,  1 \leq k \leq n - 1,$ it is easy to check that
$$1 <  x  < \frac{	x^2}{q_1}< \frac{x^3}{q_1^{2}q_2} < \cdots < 
\frac{x^j}{q_1^{j-1}q_2^{j-2} \cdot \ldots \cdot q_{j-1} }$$
and 
$$\frac{x^j}{q_1^{j-1}q_2^{j-2} \cdot \ldots \cdot q_{j-1} }  > 	
\frac{x^{j+1}}{q_1^{j} q_2^{j-1} \cdot \ldots \cdot q_{j-1}^{2}q_{j}}
> \cdots > \frac{ x^{n}}{q_1^{n-1}q_2^{n-2} \cdot \ldots \cdot  q_{n-2}^{2} q_{n-1} }.$$  
We have
$$(-1)^j Q(x) = \sum_{k=0}^{j-2} \frac{(-1)^{j+k}x^k}{q_1^{k-1}q_2^{k-2}
\cdot \ldots \cdot q_{k-1}}  + \left( -\frac{x^{j-1}}{q_1^{j-2}q_2^{j-3} \cdot 
\ldots \cdot q_{j-2}}\right.$$
$$+ \left. \frac{x^j}{q_1^{j-1}q_2^{j-2} \cdot \ldots \cdot q_{j-1}} - 
 \frac{x^{j+1}}{q_1^j q_2^{j-1} \cdot \ldots \cdot q_{j}}   \right)+ 
\sum_{k=j+2}^{n} \frac{(-1)^{j+k}x^k}{q_1^{k-1}q_2^{k-2} 
\cdot \ldots \cdot q_{k-1}} $$
$$=: \Sigma_1(x) + g(x) + \Sigma_2(x).$$

We note that the terms in $\Sigma_1(x)$ are alternating in sign and increasing in moduli,
wherein the last summand (for $k=j-2$) is positive, 
whence  $\Sigma_1(x) \geq 0$. Analogously, the summands in $\Sigma_2(x)$ are
alternating in sign and their moduli are decreasing, wherein the first summand 
(for $k=j+2$) is positive,  whence  $\Sigma_2 (x)\geq 0.$ Thus,
\begin{equation}
\label{m3}
(-1)^j Q(x)  \geq g(x) \  \mbox{for all} \  x \in \left( \frac{a_{j-1}}{a_{j}},  
\frac{a_j}{a_{j+1}} \right).
\end{equation}
So we obtain   
$$ g(x) = \frac{x^{j-1}}{q_1^{j-2}q_2^{j-3} \cdot 
\ldots \cdot q_{j- 2}} \cdot
\left( - 1 + \frac{x}{q_1q_2 \cdot \ldots \cdot q_{j-1}} - 
 \frac{x^{2}}{q_1^2 q_2^{2} \cdot \ldots \cdot   
 q_{j-2}^2 q_{j-1}^2 q_{j}}   \right).
  $$
Set $x_j = q_1q_2 \cdot \ldots \cdot q_{j-1}\sqrt{q_{j}} 
\in (q_1q_2 \cdot \ldots \cdot q_{j-1}, q_1q_2 \cdot \ldots \cdot q_{j-1} q_{j}) = 
\left( \frac{a_{j-1}}{a_{j}},  \frac{a_j}{a_{j+1}}\right) $ and obtain
$$ g(x) = \frac{x_j^{j-1}}{q_1^{j-2}q_2^{j-3} \cdot 
\ldots \cdot q_{j- 2}} \cdot
(- 1 + \sqrt{q_{j}} -  1) \geq 0,  $$
since, by our assumptions, $q_j \geq 4.$
Proposition~\ref{prop:anya}  is proved. 
\end{proof} 

For the sequence of real numbers $(b_0, b_1, \ldots , b_m)$, let us 
denote by $\nu(b_0, b_1, \ldots , b_m)$ the number of sign changes 
in this sequence (when counting the sign changes we omit zero terms).

\medskip
An immediate consequence of Proposition~\ref{prop:anya} is the following statement. 

\begin{corollary}\label{cor:anya}
{\rm 
 Let $P(x)=\sum_{k=0}^{n} a_k x^k \in 
\mathbb{R}_{+}[x] $ and suppose that $q_k(P) \geq 1$ for all
$k,  1 \leq k \leq n - 1.$ Suppose that there exist a sequence of indices  
$  1 \leq j_1 < j_2 < j_k \leq n - 1$ such that
 $q_{j_s}(P) \geq 4 $ for $1 \leq s \leq k.$  Then 
 the number of real roots of $P$ counting multiplicities is not less than
$$\nu\left((-1)^0, (-1)^{j_1}, (-1)^{j_2}, \ldots , (-1)^{j_k}, (-1)^n\right).$$
}
\end{corollary} 

\begin{proof} To prove the statement, we apply Proposition~\ref{prop:anya}, and for all 
$s,  1 \leq s \leq k,$ we find the sequence of points
$x_{j_s} \in \left(- \frac{a_{j_s}}{a_{j_{s+1}}}, - \frac{a_{j_{s-1}}}{a_{j_s}}\right)$
such that $(-1)^{j_s} P(x_{j_s}) \geq 0.$ It remains to mention that, since
 $q_k(P) \geq 1$ for all $k,  1 \leq k \leq n - 1,$  we have
 $$\frac{a_0}{a_1} \leq  \frac{a_1}{a_2} \leq \frac{a_2}{a_3} \leq \ldots \leq 
 \frac{a_{n-1}}{a_n}, $$
 whence we obtain
 $$ - \infty < x_{j_k} <   x_{j_{k-1}} < \ldots <  x_{j_2} <  x_{j_1} <0, $$
 and
 $$P(0) >0, (-1)^{j_1} P(x_{j_1}) \geq 0, (-1)^{j_2} P(x_{j_2}) \geq 0, \ldots, 
 (-1)^{j_{k-1}} P(x_{j_{k-1}}) \geq 0, $$ 
 $$(-1)^{j_k} P(x_{j_k}) \geq 0, (-1)^n P(- \infty) \geq 0.    $$ 
Corollary~\ref{cor:anya} is proved.
\end{proof} 

\section {Related conjectures  and some counterexamples}\label{sec:app2}

The following conjectures closely related to the topic under consideration were the original motivation for our study. However in the process of working on this paper we were able to disprove two of them. For the sake of completeness let us present them as well as our counterexamples. 

\medskip  

\begin{conjecture}[See Conjecture 8 of \cite{Sh}] Let $P(x) = \sum_{k=0}^n a_k x^k$ be a polynomial with positive coefficients, and consider the related (weighted) tropical polynomial
$$ftrop(P) = \max_k \left(\log(a_k) + kt + \log \binom{n}{k}\right).$$

The number of real zeros of $P (x)$ does not exceed the number of points in the tropical variety defined by $ftrop(t)$, i.e., the number of corners of the piecewise-linear continuous  function $ftrop(t), t\in \bR$.
\end{conjecture}

\begin{conjecture}[See Conjecture 9 of \cite{Sh}]\label{conj2}  Let $P (x) = \sum_{k=0}^n a_k x^k$ be a polynomial with positive coefficients. Consider the differences
$$\tilde  c_k= ( k + 1 ) a_k^2-k a_{k-1} a_{k +1},$$
where $a_{-1} =a_{n+1} =0$.  Let $0=k_1 <k_2 <\cdots <k_m =n$ be the sequence of all indices $k_i$ such that $\tilde c_{k_i}$ is positive and let $v(P)$ be the number of changes in the (binary) sequence 
$\{k_i \mod 2\}_{i=0}^m$. 

The number of real zeros of $P (x)$ does not exceed $v( P ).$ 
\end{conjecture}

\begin{conjecture}[See Conjecture 10 of \cite{Sh}]\label{conj3}  Let $P(x) =\sum_{k=0}^n a_k x^k$  be a polynomial with positive coefficients. Consider the differences
$$c_k = a_k^2-a_{k-1}a_{k + 1},\quad k=0,1,\dots, n$$
where $a_{-1} =a_{n+1} =0$.  Let $0=k_1 <k_2 <\cdots<k_m =n$ be the sequence of all
 indices such that $c_{k_i}$ is non-negative, and let $v( P )$  be the number of changes in the
sequence $\{k_i \mod 2\}_{i=0}^m$. 

The number of real zeros of  $P (x)$ does not exceed $v( P).$ 
\end{conjecture} 

\medskip
Next we provide counterexamples to Conjectures~\ref{conj2} and \ref{conj3}. Consider the polynomial

$$Q_{15}(x)=(x+1)^{13}\left(\frac{19}{360360}x^2-\frac{89}{720720}x+\frac{83}{720720}\right)=\frac{19}{360360}x^{15}+\frac{9}{16016}x^{14}+\frac{3}{1144}x^{13}$$
$$+\frac{1}{144}x^{12}+\frac{1}{88}x^{11}+\frac{1}{80}x^{10}+\frac{1}{72}x^9+\frac{3}{112}x^8+\frac{3}{56}x^7+\frac{11}{144}x^6+\frac{3}{40}x^5+\frac{9}{176}x^4+\frac{19}{792}x^3+\frac{17}{2288}x^2$$
$$+\frac{x}{728}+\frac{83}{720720}.$$

\medskip
Obviously, $\deg Q_{15}(x)=15$ and its number of real roots in 13. One has

$$\begin{cases}
q_{10}=\frac{88\cdot 72}{80^2}=\frac{99}{100}<1\\
q_{9}=\frac{88\cdot 112}{3\cdot 72^2}=\frac{140}{243}<1\\
q_{8}=\frac{72\cdot 3^2\cdot 56}{3\cdot 112^2}=\frac{27}{28}<1.
\end{cases}
$$

\medskip
Observe that  in the sequence of indices $0=k_1 <k_2 <\cdots<k_m =n$ appearing in Conjecture~\ref{conj3} each $k_j$ corresponds exactly to the case $q_{k_j}\ge 1$. Therefore among $(k_1,k_2,\dots, k_m)$ the values $8,9,10$ are missing. Thus the number of changes in the
sequence $\{k_i \mod 2\}_{i=0}^m$ is less than or equal to $15-4=11$. But $Q_{15}(x)$ has $13$ real zeros which is a contradiction with Conjecture~\ref{conj3} in degree $15$. 

To obtain counterexamples in higher degrees let us introduce the sequence of primitives $\{Q_n(x)\}_{n=15}^\infty$ such that $Q_{15}(x)$ is given above and for $n>15$
we set 
$$Q_{n}(x)=\int_{-1}^x Q_{n-1}(t)dt.$$
One easily checks that all coefficients of $Q_n(x)$ are positive and the multiplicity of the real root at $-1$ increases provided that the number of real roots of $Q_n(x)$ equals $(n-2)$. 

\medskip
For the coefficients of the polynomial $Q_n(x)$, we obtain 
$$q_{n-5}=\frac{99}{100}\cdot \frac{10\cdot 12}{11^2}\cdot \frac{11\cdot 13}{12^2} \cdots \frac{(n-1)\cdot(n+1)}{n^2}=\frac{99}{100}\cdot \frac{10}{11}\cdot \frac{n+1}{n}\to \frac{9}{10}, \quad n\to \infty
$$
$$q_{n-6}=\frac{140}{243}\cdot \frac{9\cdot 11}{10^2}\cdot \frac{10\cdot 12}{11^2} \cdots \frac{(n-1)\cdot(n+1)}{n^2}=\frac{140}{243}\cdot \frac{9}{10}\cdot \frac{n+1}{n}\to \frac{14}{27}, \quad n\to \infty
$$
$$q_{n-7}=\frac{27}{28}\cdot \frac{8\cdot 10}{9^2}\cdot \frac{9\cdot 11}{10^2} \cdots \frac{(n-1)\cdot(n+1)}{n^2}=\frac{27}{28}\cdot \frac{8}{9}\cdot \frac{n+1}{n}\to \frac{6}{7}, \quad n\to \infty.
$$

\medskip
Moreover, every $q_{n-5}, q_{n-6}, q_{n-7}$ given above is strictly less than $1$. By the same reason as above the number of sign changes in $\{k_i \mod 2\}_{i=0}^m$ is less than or equal to $n-4$ while the number of real roots of $Q_n(x)$ is $n-2$. Contradiction.  

\medskip

The same sequence $\{Q_n(x)\}$ provides counterexamples to Conjecture~\ref{conj2} for all sufficiently large $n$. Indeed, we have that the additional factor $\frac{k}{k+1}$ tends to $1$ for $k\to \infty$. In particular, we get $q_{n-5}<\frac{n-5}{n-4}$,  $q_{n-6}<\frac{n-6}{n-5}$,  $q_{n-7}<\frac{n-7}{n-6}$  for all large $n$. The same argument as above provides the required counterexamples.

%

\section{Final remark and outlook} \label{sec:outlook}

\noindent
{\bf 1.} Our proofs of the main results are  obtained by ``brute force". We are currently looking for their more conceptual proofs and generalizations of the Newton inequalitites and the Hutchinson theorem. 

\medskip
\noindent
{\bf 2.} Theorems~\ref{th:1} and  \ref{th:3} provide relevant inequalities on the sum of two $q_i$'s. Such a domain will not transform into a polytope under the logarithimic map. Maybe to obtain sharper results more sophisticate domains should be used  in the main Problem~\ref{prob:main}.

\medskip
\noindent
{\bf 3.} For the set of polynomials with positive coefficients and all real (negative) roots one can show that there exist and unique minimal with respect to inclusion inscribed and  circumscribed polytopes, see \cite{KoSh, PRS}. In the paper \cite{PRS} one can find connection of this topic with amoebas, resultants and multiplier sequences. It seems that for the majority of domains formed by polynomials with the number of real roots exceeding or below some given number of roots there is no uniqueness of inscribed and/or circumscribed polygons.  This question seems to be very important in this area of research and should be answered.


\begin{thebibliography}{8}




\bibitem[FNS]{FNS} J.~Forsg\aa rd, D.~Novikov, B.~Shapiro,  A Tropical Analog of Descartes’ Rule of Signs, International Mathematics Research Notices, Vol. 2017, No. 12, pp. 3726--3750.




\bibitem[Hu]{Hu} J. I. Hutchinson, On a remarkable class of entire functions, Trans. Amer. Math. Soc. 25
(1923), pp. 325–332.


\bibitem[KV]{KV} O.~M.~Katkova, A.~M. Vishnyakova, A remark about  positive polynomials, Mathematical Inequalities
and Applications, dx.doi.org/10.7153/mia-13-54. 

\bibitem[Ko]{Ko} V.~Kostov, Univariate polynomials and the contractibility of certain sets, Ann. Sofia Univ.  Vol. 107 (2020), Fac. Math and Inf.  

\bibitem[KoSh]{KoSh} V.~Kostov, B.Shapiro, Hardy-Petrovich-Hutchinson's problem and partial theta function,  Duke Math. J. vol 162,  issue 5 (2013) 889--924.

\bibitem[Ku]{Ku} D. C. Kurtz, A sufficient condition for all roots of a polynomial to be real, Amer. Math.
Monthly 99, no. 3 (1992) 259–263.


\bibitem[LuWe]{LuWe} A.~T.~Lundell and S.~Weingram,  The Topology of CW complexes, The University Series in Higher Mathematics. Van Nostrand Reinhold Co., New York, 1969. viii+216 pp. 


\bibitem [PRS]{PRS} M.~Passare, J.~M.~Rojas,  B.~Shapiro,  New multiplier sequences via discriminantal amobae, Moscow  Math. J., 
Volume 11, Number 3, July–September 2011, Pages 1–14. 







\bibitem[Sh]{Sh} B.~Shapiro, Problems with polynomials - the good, the bad, and the ugly, Arnold Mathematical Journal: vol 1, issue 1 (2015),  91--99.  


\end{thebibliography}
\end{document}